 \documentclass[cls,onecolumn,11pt]{IEEEtran}

\usepackage{cite}      
\usepackage{graphicx}  
\usepackage{psfrag}    
\usepackage{subfigure} 
\usepackage{url}       
\usepackage{amssymb,amsmath,amscd,amsfonts,amstext}
\usepackage{array}
\usepackage{multirow}

\usepackage[latin1]{inputenc}
\usepackage[french]{babel}

 \newtheorem{remark}{Remark}
 \newtheorem{prop}{Proposition}

 \newtheorem{theo}{Theorem}
 \newtheorem{lemma}{Lemma}

 \newtheorem{coro}{Corollary}
 
 \newtheorem{definition}{Definition}

 \newcommand{\tr}[1]{\textrm{Tr}\left(#1\right)}
 \newcommand{\mR}{{\mathbb{R}}}

 \newcommand{\mC}{{\mathbb{C}}}

 \newcommand{\real}[1]{{\mathfrak{R}}\left[#1\right]}
 \newcommand{\imag}[1]{{\mathfrak{I}}\left[#1\right]}

 \newcommand{\PP}[1]{{\mathbb{P}}\left\{ #1 \right\}}

 \newcommand{\bs}{\boldsymbol}
 \newcommand{\dsp}{\displaystyle}

\IEEEoverridecommandlockouts 

\begin{document}

\title{Asymptotic Independence in the Spectrum of the Gaussian Unitary Ensemble}
\author{P. Bianchi, M. Debbah and J. Najim\\
\today
\thanks{This work was partially supported by the ``Agence Nationale de la Recherche'' via the program 
``Masses de donn{é}es et connaissances'', project SESAME and by the ``GDR ISIS'' via 
the program ``jeunes chercheurs''.}
\thanks{P. Bianchi and M. Debbah are with SUPELEC, 
France. {\tt  \{pascal.bianchi,merouane.debbah\}@supelec.fr}\ ,}  
\thanks{M. Debbah holds the Alcatel-Lucent Chair on flexible radio,}
\thanks{J. Najim is with CNRS and Télécom ParisTech, Paris, France.
        {\tt najim@enst.fr}\ }
}
\maketitle

\bibliographystyle{plain}

\begin{abstract}
  Consider a $n× n$ matrix from the Gaussian Unitary Ensemble
  (GUE). Given a finite collection of bounded disjoint real Borel sets
  $(\Delta_{i,n},\ 1\leq i\leq p)$, properly rescaled, and eventually
  included in any neighbourhood of the support of Wigner's semi-circle
  law, we prove that the related counting measures $({\mathcal
    N}_n(\Delta_{i,n}),\ 1\leq i\leq p)$, where ${\mathcal N}_n(\Delta)$
  represents the number of eigenvalues within $\Delta$, are
  asymptotically independent as the size $n$ goes to infinity, $p$
  being fixed.
  
  As a consequence, we prove that the largest and smallest
  eigenvalues, properly centered and rescaled, are asymptotically
  independent; we finally describe the fluctuations of the condition
  number of a matrix from the GUE.

\end{abstract}

\section{Introduction and main result}

Denote by ${\cal H}_n$ the set of $n×n$ random Hermitian matrices endowed with the probability measure  
$$
P_n(d\,{\bf M}) := Z_n^{-1} \exp\left\{ -\frac{n}{2}\: \tr{{\bf M}^2}\right\} d\,{\bf M}\ ,
$$
where $Z_n$ is the normalization constant and where
$$
d\,{\bf M} = \prod_{i=1}^n d \,M_{ii} \prod_{1\leq i<j\leq n} \real{d\,M_{ij}}\prod_{1\leq i<j\leq n} \imag{d\,M_{ij}}
$$
for every ${\bf M}=(M_{ij})_{1\leq i,j\leq n}$ in ${\cal H}_n$
($\real{z}$ being the real part of $z\in \mathbb{C}$ and $\imag{z}$ its imaginary part). This set is
known as the Gaussian Unitary Ensemble (GUE) and corresponds to the
case where a $n× n$ hermitian matrix ${\bf M}$ has independent, complex,
zero mean, Gaussian distributed entries with variance
$\mathbb{E}|M_{ij}|^2=\frac{1}{n}$ above the diagonal while the
diagonal entries are independent real Gaussian with the same variance.
Much is known about the spectrum of ${\bf M}$. Denote by $
\lambda_1^{(n)}, \lambda_2^{(n)}, \cdots ,\lambda_n^{(n)} $ the
eigenvalues of ${\bf M}$ (all distinct with probability one), then:
\begin{itemize}
\item The joint probability density function of the (unordered) eigenvalues
  $(\lambda_1^{(n)}, \cdots, \lambda_n^{(n)})$ is given by:
$$
p_n(x_1,\cdots, x_n)= C_n e^{-\frac{\sum x_i^2}2} \prod_{j<k} |x_j -x_k|^2\ ,
$$
where $C_n$ is the normalization constant. 
\item \cite{Wig58} The empirical distribution of the eigenvalues $\frac 1n
  \sum_{i=1}^n \delta_{\lambda_i^{(n)}}$ ($\delta_x$ stands for the
  Dirac measure at point $x$) converges toward Wigner's semi-circle law
  whose density is:
$$
\frac{1}{2\pi} \boldsymbol{1}_{(-2,2)}(x) \sqrt{4-x^2} \ .
$$
\item \cite{BaiYin88} The largest eigenvalue $\lambda_{\max}^{(n)}$ (resp.  the smallest
  eigenvalue $\lambda_{\min}^{(n)}$) almost surely converges to $2$ (resp.
  $-2$), the right-end (resp. left-end) point of the support of the semi-circle law as
  $n\to\infty$.
\item \cite{TraWid94} The centered and rescaled quantity $n^{\frac 23}
  \left( \lambda^{(n)}_{\max} -2\right)$ converges in distribution toward
  Tracy-Widom distribution function $F^+_{GUE}$ which can be defined
  in the following way:
$$
F^+_{GUE}(s) =\exp\left( -\int_s^{\infty} (x-s)q^2(x)\,dx\right)\ , 
$$
where $q$ solves the Painlevé II differential equation:
$$
\begin{array}{l}
q''(x) = xq(x) + 2 q^3(x)\ ,\\
q(x) \sim \textrm{Ai}(x) \quad \textrm{as} \quad x\to \infty
\end{array}
$$
and Ai$(x)$ denotes the Airy function. In particular, $F_{GUE}^+$ is
continuous. Similarly, $n^{\frac 23} \left( \lambda^{(n)}_{\min} +2\right)
\xrightarrow{\mathcal D} F^-_{GUE}$ where $$F_{GUE}^-(s) =
1-F_{GUE}^+(-s)\ .$$
\end{itemize} 

If $\Delta$ is a Borel set in $\mathbb{R}$, denote by:
$$
{\mathcal N}_n(\Delta)=\# \left\{ \lambda^{(n)}_i \in \Delta\right\},
$$
i.e. the number of eigenvalues in the set $\Delta$. The following theorem
is the main result of the article. 

\begin{theo}\label{theo:main} Let ${\bf M}$ be a $n× n$ matrix from the GUE with eigenvalues
  $(\lambda_1^{(n)}, \cdots, \lambda_n^{(n)})$. Let $p\geq 2$ be an
  integer and let $(\mu_1,\cdots,\mu_p)\in \mathbb{R}^p$ be such that
  $ -2=\mu_1 <\mu_2<\cdots<\mu_p=2$.  Denote by ${\bs
    \Delta}=(\Delta_{1},\cdots, \Delta_{p})$ a collection of $p$
  bounded Borel sets in $\mR$ and consider ${\bs
    \Delta}_n=(\Delta_{1,n},\cdots, \Delta_{p,n})$ defined by the
  following scalings:
\begin{eqnarray*}
(edge)\qquad  \Delta_{1,n} &:=& -2 +\frac{\Delta_1}{n^{2/3}}\ ,\qquad \Delta_{p,n} \quad :=\quad 
2 +\frac{\Delta_p}{n^{2/3}}\ ,\\
(bulke)\qquad \Delta_{i,n} &:=& \mu_i +\frac{\Delta_i}{n}\ ,\qquad 2\leq i\leq p-1\ .
\end{eqnarray*}
Let $(\ell_1,\cdots,\ell_p)\in \mathbb{N}^p$, then:
$$
\lim_{n\to\infty} \left( \mathbb{P} \left( {\mathcal
      N}_n(\Delta_{1,n})=\ell_1,\cdots, {\mathcal
      N}_n(\Delta_{p,n})=\ell_p\right) - \prod_{k=1}^{p} \mathbb{P} \left( {\mathcal
      N}_n(\Delta_{k,n})=\ell_k\right)\right) = 0\ .
$$
\end{theo}

Proof of Theorem \ref{theo:main} is postponed to Sections
\ref{sec:proof}. In Section \ref{sec:asymptotic}, we state and prove
the asymptotic independence of the random variables $n^{\frac
  23}\left(\lambda_{\min}^{(n)}+2\right)$ and $n^{\frac 23} \left(
  \lambda^{(n)}_{\max} -2\right)$, where $\lambda_{\min}^{(n)}$ and
$\lambda_{\max}^{(n)}$ are the smallest and largest eigenvalues of
${\bs M}$. We then describe the asymptotic fluctuations of the ratio
$\frac {\lambda^{(n)}_{\max}}{\lambda^{(n)}_{\min}}$.

\subsection*{Acknowlegment} The authors are grateful to Walid Hachem
for fruitful discussions and many helpful remarks; the authors wish
also to thank Eric Amar for useful references related to complex analysis.

\section{Asymptotic independence of extreme eigenvalues}
\label{sec:asymptotic}

In this section, we prove that the random variables $n^{\frac 23}
\left( \lambda_{\max}^{(n)} -2\right)$ and $n^{\frac
  23}\left(\lambda_{\min}^{(n)}+2\right)$ are asymptotically
independent as the size of matrix ${\bs M}$ goes to infinity. We then
apply this result to describe the fluctuations of
$\frac{\lambda_{\max}^{(n)}}{\lambda_{\min}^{(n)}}$.  In the sequel,
we drop the upperscript $^{(n)}$ to lighten the notations.

\subsection{Asymptotic independence}
Specifying $p=2$, $\mu_1=-2$, $\mu_2=2$ and getting rid of the
boundedness condition over $\Delta_1$ and $\Delta_2$ in Theorem
\ref{theo:main} yields the following:
\begin{coro}\label{coro:independence} Let ${\bs M}$ be a $n× n$ matrix from the GUE. 
  Denote by $\lambda_{\min}$ and $\lambda_{\max}$ its smallest and
  largest eigenvalues, then the following holds true:
\begin{multline*}
\mathbb{P}\left( 
n^{\frac
    23}\left(\lambda_{\min}+2\right)<x\ ,\  
n^{\frac
  23} \left( \lambda_{\max} -2\right)<y \right) \\
- \mathbb{P}\left(n^{\frac
    23}\left(\lambda_{\min}+2\right)<x \right) \mathbb{P}\left(n^{\frac
  23} \left( \lambda_{\max} -2\right)< y  \right)
\xrightarrow[n\to\infty]{} 0\ . 
\end{multline*}
Otherwise stated, 
$$
\left(n^{\frac 23} (\lambda_{\min}+2) ,n^{\frac 23}(\lambda_{\max}-2) \right) \xrightarrow[n\to \infty]{\mathcal D} (\lambda_-, \lambda_+), 
$$
where $\lambda_-$ and $\lambda_+$ are independent random
variables with distribution functions $F^-_{GUE}$ and $F^+_{GUE}$.
\end{coro}
\begin{proof} 
Denote by $(\lambda_{(i)})$ the ordered eigenvalues of
  ${\bs M}$: $\lambda_{\min}= \lambda_{(1)}\leq
  \lambda_{(2)}\leq\cdots\leq \lambda_{(n)}=\lambda_{\max}$.

 Let $(x,y)\in \mathbb{R}^2$ and take $\alpha \geq \max(|x|,|y|)$. Let $\Delta_1=(-\alpha,x)$
  and $\Delta_2=(y,\alpha)$ so that
$$
\Delta_{1,n} \ =\  \left( -2 -\frac{\alpha}{n^{\frac 23}}, -2 +\frac{x}{n^{\frac 23}} \right)\quad\textrm{and}\quad 
\Delta_{2,n} \ =\ \left( 2 +\frac{y}{n^{\frac 23}}, 2 +\frac{\alpha}{n^{\frac 23}} \right)\ .
$$
We have:
\begin{eqnarray}
\left\{ {\mathcal N}(\Delta_{1,n})=0 \right\} &=& \left\{n^{\frac
    23}(\lambda_{\min} +2)> x\right\}\cup \left\{ \exists
  i\in\{1,\cdots,n\};\ \lambda_{(i)} \leq -2-\frac \alpha{n^{\frac 23}}\
  ,\ \lambda_{(i+1)} \geq -2+\frac x{n^{\frac 23}} \right\}\nonumber \\
&:=& \left\{n^{\frac
    23}(\lambda_{\min} +2)> x\right\}\cup \ \{\, \Pi(-\alpha,x)\, \}\ , \label{delta1}
\end{eqnarray}
with the convention that if $i=n$, the condition simply becomes
$\lambda_{\max} \leq -2-\alpha{n^{-\frac 23}}$. Note that
both sets in the right-hand side of the equation are
disjoint. Similarly:
\begin{eqnarray}\nonumber
\left\{ {\mathcal N}(\Delta_{2,n})=0 \right\} &=& \left\{ n^{\frac
    23}(\lambda_{\max} -2)<y\right\}\cup 
\left\{ \exists
  i\in\{1,\cdots,n\};\ \lambda_{(i-1)} \leq 2+\frac y{n^{\frac 23}}\
  ,\ \lambda_{(i)} \geq 2+\frac \alpha{n^{\frac 23}} \right\}\ ,\\
&:=& \left\{ n^{\frac
    23}(\lambda_{\max} -2)<y\right\}\cup \{\, \tilde \Pi(y,\alpha)\,\} \label{delta2}\ ,
\end{eqnarray}
with the convention that if $i=1$, the condition simply becomes
$\lambda_{\min} \geq 2+\alpha{n^{-\frac 23}}$. Gathering the two previous
equalities enables to write $\{ {\mathcal
    N}(\Delta_{1,n})=0 , {\mathcal N}(\Delta_{2,n})=0 \}$ as the following
union of disjoint events:
\begin{multline}\label{delta12}
\left\{ {\mathcal
    N}(\Delta_{1,n})=0 \ ,\ {\mathcal N}(\Delta_{2,n})=0 \right\}\\
=\ \left\{ \Pi(-\alpha,x)\ ,\ n^{\frac 23}(\lambda_{\max} -2)<y\right\}
\cup \left\{ \Pi(-\alpha,x)\ ,\ \tilde \Pi(y,\alpha) \right\}
\cup \left\{ n^{\frac 23}(\lambda_{\min} +2)>x\ ,\tilde \Pi(y,\alpha) \right\}\\
\cup \left\{ n^{\frac 23}(\lambda_{\min} +2)>x\ ,\  n^{\frac 23}(\lambda_{\max} -2)<y\right\}\ . 
\end{multline} 
Define:
\begin{eqnarray}
  u_n&:=& \mathbb{P}\left\{ n^{\frac 23}(\lambda_{\min} +2)>x\ ,\  n^{\frac 23}(\lambda_{\max} -2)<y\right\}
-\mathbb{P}\left\{ n^{\frac 23}(\lambda_{\min} +2)>x\right\}\mathbb{P}\left\{n^{\frac 23}(\lambda_{\max} -2)<y\right\}\nonumber\\
&=& \mathbb{P}\left\{ {\mathcal N}(\Delta_{1,n})=0 \ ,\ {\mathcal N}(\Delta_{2,n})=0\right\}
- \mathbb{P}\left\{ {\mathcal N}(\Delta_{1,n})=0 \right\}\mathbb{P}\left\{ {\mathcal N}(\Delta_{2,n})=0 \right\} + \epsilon_n(\alpha)\:,
\label{eq:def_un}
\end{eqnarray}
where, by equations~(\ref{delta1}), (\ref{delta2}) and~(\ref{delta12}),
\begin{multline*}
\epsilon_n(\alpha) := - \mathbb{P}\left\{ \Pi(-\alpha,x)\ ,\ n^{\frac 23}(\lambda_{\max} -2)<y\right\}
- \mathbb{P}\left\{ \Pi(-\alpha,x)\ ,\ \tilde \Pi(y,\alpha) \right\}\\
-\mathbb{P} \left\{ n^{\frac 23}(\lambda_{\min} +2)>x\ ,\tilde \Pi(y,\alpha) \right\}
+\mathbb{P}\left\{ {\mathcal N}(\Delta_{1,n})=0 \right\}\mathbb{P}\left\{ \tilde \Pi(y,\alpha)\right\}\\
+\mathbb{P}\left\{\Pi(-\alpha,x)\right\}\mathbb{P}\left\{ {\mathcal N}(\Delta_{2,n})=0 \right\}
-\mathbb{P}\left\{\Pi(-\alpha,x)\right\}\mathbb{P}\left\{ \tilde \Pi(y,\alpha)\right\}\:.
\end{multline*} 
Using the triangular inequality, we obtain:
$$
|\epsilon_n(\alpha)| \leq 6 \max\left( \mathbb{P}\left\{ \Pi(-\alpha,x)\right\} , \mathbb{P}\left\{ \tilde \Pi(y,\alpha)\right\}\right)
$$
As $\{\, \Pi(-\alpha,x)\, \}\subset\{n^{\frac 23}(\lambda_{\min}+2)<-\alpha\}$, we have
$$
\mathbb{P}\{ \Pi(-\alpha,x)\ \}\leq \mathbb{P}\{n^{\frac 23}(\lambda_{\min}+2)<-\alpha\}\xrightarrow[n\to\infty]{} F_{GUE}^-(-\alpha)
\xrightarrow[\alpha\to\infty]{} 0\ . 
$$
We can apply the same arguments to $\{\, \tilde \Pi(y,\alpha)\, \}\subset\{n^{\frac 23}(\lambda_{\max}-2)>\alpha\}$. We thus obtain:
\begin{equation}\label{eq:eps-to-zero}
\lim_{\alpha\to\infty}\limsup_{n\to\infty} |\epsilon_n(\alpha)| = 0\:.
\end{equation}
The difference $ \mathbb{P}\left\{ {\mathcal N}(\Delta_{1,n})=0 \ ,\ {\mathcal N}(\Delta_{2,n})=0\right\}
- \mathbb{P}\left\{ {\mathcal N}(\Delta_{1,n})=0 \right\}\mathbb{P}\left\{ {\mathcal N}(\Delta_{2,n})=0 \right\}$
in the right­hand side of~(\ref{eq:def_un}) converges to zero as $n\to\infty$ by Theorem~\ref{theo:main}.
We therefore obtain:
$$
\limsup_{n\to \infty} |u_n| = \limsup_{n\to\infty} |\epsilon_n(\alpha)|\:.
$$
The lefthand side of the above equation is a constant w.r.t. $\alpha$
while the second term (whose behaviour for small $\alpha$ is unknown)
converges to zero as $\alpha\to\infty$ by \eqref{eq:eps-to-zero}. Thus, $\lim_{n\to \infty}
u_n=0$. The mere definition of $u_n$ together with Tracy and Widom fluctuation results yields:
$$
\lim_{n\to\infty} \mathbb{P}\left\{ n^{\frac 23}(\lambda_{\min} +2)>x\ ,\  n^{\frac 23}(\lambda_{\max} -2)<y\right\}
=\left(1-F_{GUE}^-(x)\right) × F_{GUE}^+(y)\:.
$$
This completes the proof of Corollary~\ref{coro:independence}.

\end{proof}
\subsection{Application: Fluctuations of the condition number in the GUE}
As a simple consequence of Corollary \ref{coro:independence}, we can easily describe the fluctuations 
of the condition number $\frac{\lambda_{\max}}{\lambda_{\min}}$. 
\begin{coro}\label{coro:condition} Let ${\bs M}$ be a $n× n$ matrix from the GUE. 
  Denote by $\lambda_{\min}$ and $\lambda_{\max}$ its smallest and
  largest eigenvalues, then:
$$
n^{\frac 23} \left( \frac{\lambda_{\max}}{\lambda_{\min}} +1\right)
\xrightarrow[n\to\infty]{\mathcal D} -\frac 12 (\lambda_-+\lambda_+)\ ,
$$
where $\xrightarrow{\mathcal D}$ denotes convergence in distribution, $\lambda_-$ and $\lambda_+$
are independent random variable with respective distribution $F_{GUE}^-$ and $F_{GUE}^+$.
\end{coro}

\begin{proof}
The proof is a mere application of Slutsky's lemma (see for instance \cite[Lemma 2.8]{Van98}).
Write:
\begin{multline}\label{slutsky}
n^{\frac 23} \left( \frac{\lambda_{\max}}{\lambda_{\min}} +1\right)= -\frac 12 \left[
n^{\frac 23}(\lambda_{\max}-2) + n^{\frac 23}(\lambda_{\min}+2) \right]\\
+\frac{\lambda_{\min} +2}{2\lambda_{\min}}  \left[
n^{\frac 23}(\lambda_{\max}-2) + n^{\frac 23}(\lambda_{\min}+2) \right]\ .
\end{multline}
Now, $\frac{\lambda_{\min} +2}{2\lambda_{\min}} $ goes almost surely
to zero as $n\to\infty$, and $n^{\frac 23}(\lambda_{\max}-2) +
n^{\frac 23}(\lambda_{\min}+2)$ converges in distribution to the
convolution of $F_{GUE}^-$ and $F_{GUE}^+$ by Corollary \ref{coro:independence}. Thus, Slutsky's lemma implies that
$$
\frac{\lambda_{\min} +2}{2\lambda_{\min}}  \left[
n^{\frac 23}(\lambda_{\max}-2) + n^{\frac 23}(\lambda_{\min}+2) \right]\xrightarrow[n\to\infty]{\mathcal D} 0.
$$
Another application of Slutsky's lemma yields the convergence (in
distribution) of the right-hand side of \eqref{slutsky} to the limit
of $-\frac 12 \left[ n^{\frac 23}(\lambda_{\max}-2) + n^{\frac
    23}(\lambda_{\min}+2) \right]$, that is $-\frac 12 (X+Y)$ with $X$
and $Y$ independent and distributed according to $F_{GUE}^-$ and $F_{GUE}^+$. Proof of Corollary \ref{coro:condition}
is completed.
\end{proof}
\section{Proof of Theorem \ref{theo:main}}\label{sec:proof}
\subsection{Useful results}\label{sec:proof1}
\subsubsection{Kernels}
Let $\{ H_k(x)\}_{k\geq 0}$ be the classical Hermite polynomials
$H_k(x) := e^{x^2}\left(-\frac{d}{dx}\right)^k e^{-x^2}$ and consider
the function $\psi_k^{(n)}(x)$ defined for $0\leq k\leq n-1$ by:
$$
\psi_k^{(n)}(x) :=
\left(\frac{n}{2}\right)^{\frac{1}{4}}\frac{e^{-\frac{nx^2}{4}}}{(2^k
  k! \sqrt{\pi})^{\frac{1}{2}}} H_k\left(\sqrt{\frac n2} x\right)\ .
$$
Denote by $K_n(x,y)$ the following Kernel on $\mR^2$:
\begin{eqnarray}
  K_n(x,y) &:=& \sum_{k=0}^{n-1} \psi_k^{(n)}(x) \psi_k^{(n)}(y)\label{eq:kernelsum} \\
  &=& \frac{\psi_n^{(n)}(x)\psi_{n-1}^{(n)}(y)-\psi_n^{(n)}(y)\psi_{n-1}^{(n)}(x)}{x-y}
  \label{eq:kernel}
\end{eqnarray}
Equation \eqref{eq:kernel} is obtained from~(\ref{eq:kernelsum}) by
the Christoffel-Darboux formula. We recall the two well-known asymptotic results 
 \begin{prop}\label{prop:convergence-kernel}
   \begin{itemize}
   \item[a)] \emph{Bulk of the spectrum.} Let $\mu\in (-2,2)$.
     \begin{equation}
       \forall (x,y)\in\mR^2, \;\lim_{n\to\infty} \frac{1}{n} K_n\left(\mu + \frac{x}{n}, \mu + \frac{y}{n}\right) =
       \frac{\sin \pi\rho(\mu)(x-y)}{\pi(x-y)}\:,
       \label{eq:Kbulk}
     \end{equation}
     where $\rho(\mu)=\frac{\sqrt{4-\mu^2}}{2\pi}
$.
     Furthermore, the convergence~\eqref{eq:Kbulk} is uniform on every compact set of $\mR^2$.
   \item[b)] \emph{Edge of the spectrum.} 
     \begin{equation}
       \forall (x,y)\in\mR^2, \;\lim_{n\to\infty} \frac{1}{n^{2/3}} K_n\left(2 + \frac{x}{n^{2/3}}, 2 + \frac{y}{n^{2/3}}\right) =
       \frac{Ai(x)Ai'(y)-Ai(y)Ai'(x)}{x-y}\:,
       \label{eq:Kedge}
     \end{equation}
     where $Ai(x)$ is the Airy function. Furthermore, the convergence~\eqref{eq:Kedge} is uniform on every compact 
     set of $\mR^2$.
   \end{itemize}
 \end{prop}
We will need as well the following result on the asymptotic behavior of functions $\psi_k^{(n)}$.
\begin{prop}
Let $\mu\in (-2,2)$, let $k=0$ or $k=1$ and denote by $K$ a compact set of $\mR$.
   \begin{itemize}
   \item[a)] \emph{Bulk of the spectrum.} 
     There exists a constant $C$ such that for large $n$,
     \begin{equation}
     \label{eq:psibulk}
     \sup_{x\in K}\left|\psi_{n-k}^{(n)}\left(\mu+\frac{x}{n}\right)\right|\leq C\:.
     \end{equation}
   \item[b)] \emph{Edge of the spectrum.} 
     There exists a constant $C$ such that for large $n$,
     \begin{equation}
     \sup_{x\in K}\left|\psi_{n-k}^{(n)}\left(± 2± \frac{x}{n^{2/3}}\right)\right|\leq n^{1/6} C\:.
       \label{eq:psiedge}
     \end{equation}
   \end{itemize}  
\end{prop}
The proof of these results can be found in \cite[Chapter 7]{Pas07}.

\subsubsection{Determinantal representations, Fredholm determinants}

There are determinantal representations using kernel $K_n(x,y)$ for the joint density
$p_n$ of the eigenvalues $(\lambda_i^{(n)};1\leq i\leq n)$, and for its
marginals (see for instance \cite[Chapter 6]{Meh04}:
\begin{eqnarray}
  \label{eq:pdfeig}
  p_n(x_1,\cdots,x_n) &=& \frac{1}{n!}\det \left\{ K_n(x_i,x_j)\right\}_{1\leq i,j\leq n}\ ,\\
\label{eq:marginals}
\int_{\mR^{n-m}} p_n(x_1,\cdots, x_n)dx_{m+1}\cdots dx_n 
&=& \frac{(n-m)!}{n!} \det\left\{ K_n(x_i,x_j) \right\}_{1\leq i,j\leq m}\quad (m\le n)\ .
\end{eqnarray}


\begin{definition}
Consider a linear operator $S$ defined for any bounded integrable function $f:\mR\to\mR$ by
$$
S f : x \mapsto \int_\mR S(x,y) f(y) dy\:,
$$
where $S(x,y)$ is a bounded integrable Kernel on $\mR^2\to \mR$ with
compact support. The Fredholm determinant $D(z)$ associated with
operator $S$ is defined as follows:
 \begin{equation}
   \label{eq:fredholm}
   \det(1-z S) := 1+\sum_{k=1}^\infty \frac{(-z)^k}{k!}\int_{\mR^k} 
\det \left\{ S(x_i,x_j)\right\}_{1\leq i,j\leq  k} dx_1\cdots dx_k\:,
 \end{equation}
for each $z\in\mC$, i.e. it is an entire function. Its logarithmic derivative has the simple expression:
\begin{equation}\label{eq:logdet}
\frac{D'(z)}{D(z)} = -\sum_{k=0}^{\infty} T(k+1) z^k\ ,
\end{equation}
where
\begin{equation}\label{eq:trace}
T(k) = \int_{\mathbb{R}^k} S(x_1,x_2) S(x_2,x_3)\cdots S(x_k,x_1)\,dx_1\cdots dx_k\ .
\end{equation}
\end{definition}
For details related to Fredholm determinants, see for instance \cite{Smi58,Tri57}. 

The following kernel will be of constant use in the sequel:
\begin{equation}
  S_n(x,y ;\bs\lambda,\bs\Delta ) = \sum_{i=1}^p \lambda_i \mathbf{1}_{\Delta_i}(x) K_n(x,y),
\label{eq:Sn}
\end{equation}
where $\bs\lambda=(\lambda_1,\cdots,\lambda_p)\in\mR^p$ or
$\bs\lambda\in\mC^p$, depending on the need.  
\begin{remark}
  Kernel $K_n(x,y)$ is unbounded and one cannot consider its
  Fredholm determinant without caution. Kernel $S_n(x,y)$ is bounded
  in $x$ since the kernel is zero if $x$ is outside the compact set
  $\cup_{i=1}^p \Delta_i$, but a priori unbounded in $y$. In all the
  forthcoming computations, one may replace $S_n$ with the bounded
  kernel $\tilde S_n(x,y)= \sum_{i,\ell=1}^p \lambda_i
  \mathbf{1}_{\Delta_i}(x) \mathbf{1}_{\Delta_{\ell}}(y) K_n(x,y)$ and
  get exactly the same results. For notational convenience, we keep on working with $S_n$. 
\end{remark}

\begin{prop}\label{prop:representation}
  Let $p\geq 1$ be a fixed integer; let ${\bs \ell} = (\ell_1,\cdots,\ell_p)\in
  \mathbb{N}^p$ and denote by $\bs\Delta = (\Delta_1,\cdots,\Delta_p)$, where
  every $\Delta_i$ is a bounded Borel set. Assume that the
  $\Delta_i$'s are pairwise disjoint. Then the following identity holds true:
  \begin{multline}\label{eq:Enl}
     \mathbb{P}\left\{ {\mathcal N}(\Delta_1)=\ell_1,\cdots,{\cal
             N}(\Delta_p)=\ell_p \right\}    \\
         = \frac{1}{\ell_1!\cdots
           \ell_p!}\left(-\frac{\partial}
 {\partial\lambda_1}\right)^{\ell_1}\cdots\left(-\frac{\partial}{\partial\lambda_p}\right)^{\ell_p}
         \det\left(1- S_n(\bs\lambda,\bs\Delta)\right)\bigg
         |_{\lambda_1=\cdots=\lambda_p=1}\:,
\end{multline}
where $S_n(\bs\lambda,\bs\Delta)$ is the operator associated to the kernel defined in \eqref{eq:Sn}.
\end{prop}
\medskip

Proof of Proposition \ref{prop:representation} is postponed to Appendix \ref{proof:representation}.

\subsubsection{Useful estimates for kernel $S_n(x,y; \bs \lambda,\bs \Delta)$ and its iterations}

Consider ${\bs \mu}$,
${\bs \Delta}$ and ${\bs \Delta}_n$ as in Theorem \ref{theo:main}.
Assume moreover that $n$ is large enough so that the Borel sets
$(\Delta_{i,n}; 1\leq i\leq p)$ are pairwise disjoint. 
For $i\in \{1,\cdots,p\}$, define $\kappa_i$ as
\begin{equation}
  \kappa_i=\left\{
  \begin{array}{l}
    1\;\textrm{ if } -2 < \mu_i < 2 \\
    \frac{2}{3}\;\textrm{ if } \mu_i=± 2
  \end{array}\right.\ ,
\label{eq:deltai}
\end{equation}
i.e. $\kappa_1=\kappa_p=\frac{2}{3}$ and $\kappa_i=1$ for $1<i<p$. 

Let $\bs \lambda\in \mC^p$. With a slight abuse of notation, denote by
$S_n(x,y;\bs \lambda)$ the kernel:
\begin{equation}\label{eq:Sn-Deltan}
S_n(x,y;\bs \lambda):= S_n(x,y;{\bs \lambda}, {\bs \Delta}_n)\ .
\end{equation}

For $1\le m,\ell \le p$, define:
\begin{eqnarray} {\cal M}_{m×\ell,n}(\bs \Lambda) &:=& \sup_{\bs
    \lambda \in \bs \Lambda} \sup_{(x,y)\in
    \Delta_{m,n}×\Delta_{\ell,n}} \left| S_n(x,y;\bs \lambda)
  \right|\:,
\end{eqnarray}
where $S_n(x,y;\bs \lambda)$ is given by \eqref{eq:Sn-Deltan}.

\begin{prop}
\label{prop:ray}
Let $\bs \Lambda \subset \mC^p$ be a
  compact set. There exist two constants $R=R(\bs \Lambda)>0$ and
  $C=C(\bs \Lambda)>0$, independent from $n$, such that for $n$ large enough, 
\begin{equation}\label{eq:ray-limcross-uniform}
\left\{
\begin{array}{lcll}
   {\cal M}_{m× m,n}({\bs \Lambda}) &\leq& \dsp R^{-1}n^{\kappa_m}\ , & 1\leq m\leq p\\
  {\cal M}_{m× \ell,n}({\bs \Lambda}) &\leq&\dsp  C n^{1-\frac{\kappa_m +\kappa_{\ell}}2} \ , & 1\leq m,\ell \leq p,\ m\neq \ell
\end{array}
\right.\ .
\end{equation}
\end{prop}
Proposition \ref{prop:ray} is proved in Appendix \ref{proof-prop:ray}.

Consider the iterated kernel $|S_n|^{(k)}(x,y;\bs \lambda)$ defined by:
\begin{equation}\label{eq:def-iterated}
\left\{
\begin{array}{ll}
|S_n|^{(1)}(x,y;\bs \lambda)= |S_n(x,y;\bs \lambda)| &\\
|S_n|^{(k)}(x,y;\bs \lambda)=\int_{\mathbb{R}^{k-1}} |S_n(x,u;\bs\lambda)|× |S_n|^{(k-1)} (u,y;\bs \lambda)\,du &k\geq 2
\end{array}
\right.\ ,
\end{equation}
where $S_n(x,y;\bs \lambda)$ is given by \eqref{eq:Sn-Deltan}. The
next estimates will be stated with $\bs \lambda \in \mC^p$ fixed. Note that
$|S_n|^{(k)}$ is nonnegative and writes:
$$
 \int_{\mR^{k-1}}|S_n(x,u_1;\bs\lambda)S_n(u_1,u_2;\bs\lambda)\cdots S_n(u_{k-1},y;\bs\lambda)|du_1\cdots du_{k-1}\ .
$$
As previously, define for $1\leq m,\ell\leq p$:
$$
  {\cal M}_{m× \ell,n}^{(k)}(\bs\lambda) := \sup_{(x,y)\in \Delta_{m,n}×\Delta_{\ell,n}} |S_n|^{(k)}(x,y;\bs \lambda)
$$
The following estimates hold true:
\begin{prop}\label{prop:estimees-M}
  Consider the compact set $\bs \Lambda = \{\bs \lambda\}$ and the
  associated constants $R=R(\bs\lambda)$ and $C=C(\bs \lambda)$ as
  given by Prop. \ref{prop:ray}. Let $\beta>0$ be such that $\beta>
  R^{-1}$ and consider $\epsilon \in (0,\frac{1}{3})$.  There exists
  an integer $N_0=N_0(\beta,\epsilon)$ such that for every $n\geq N_0$
  and for every $k\geq 1$,
\begin{equation}
  \label{eq:Knbd}
  \left\{
    \begin{array}{lcll}
       \dsp {\cal M}_{m× m,n}^{(k)}(\bs \lambda) &\leq& \dsp \beta^k n^{\kappa_m} \ , &1\leq  m\leq p\\   
       \dsp {\cal M}_{m× \ell,n}^{(k)} (\bs \lambda)&\leq& \dsp C \beta^{k-1}\ n^{\left(1+\epsilon -\frac{\kappa_m +\kappa_{\ell}}{2}\right)}\ ,
& 1\leq m,\ell\leq p,\ m\neq \ell
    \end{array}\right. \ .
\end{equation}
\end{prop}
Proposition \ref{prop:estimees-M} is proved in Appendix \ref{proof-prop:estimees-M}.

\subsection{End of proof}
\label{sec:proof2}

Consider ${\bs \mu}$, ${\bs \Delta}$ and ${\bs \Delta}_n$ as in
Theorem \ref{theo:main}.  Assume moreover that $n$ is large enough so
that the Borel sets $(\Delta_{i,n}; 1\leq i\leq p)$ are pairwise
disjoint. As previously, denote $S_n(x,y;\bs
\lambda)=S_n(x,y;\bs\lambda,\bs\Delta_n)$; denote also
$S_{i,n}(x,y;\lambda_i)=S_n(x,y;\lambda_i,\Delta_{i,n})= \lambda_i
\mathbf{1}_{\Delta_{i}}(x)K_n(x,y)$, for $1\leq i\leq p$. Note that
$S_n(x,y;\bs \lambda) =S_{i,n}(x,y;\lambda_i)$ if $x\in\Delta_{i,n}$.

For every $z\in\mC$ and $\bs \lambda\in \mC^p$, we use the following
notations:
\begin{equation}\label{eq:Dn}
D_{n}(z, \bs \lambda) := \det(1-zS_n(\bs\lambda,\bs \Delta_{n}))
\quad \textrm{and}\quad 
D_{n,i}(z, \lambda_i) := \det(1-zS_n(\lambda_i,\Delta_{i,n}))
\end{equation}

The following controls will be of constant use in the sequel.

\begin{prop}
\label{prop:DvoisZero}
\begin{enumerate}
\item Let $\bs \lambda\in \mC^p$ be fixed. The sequences of functions:
$$
z\mapsto D_{n}(z,\bs \lambda)\qquad \textrm{and} \qquad z\mapsto D_{i,n}(z,\lambda_i),\quad 1\leq i\leq p
$$
are uniformly bounded on every compact subset of $\mC$. 
\item Let $z=1$. The sequences of functions:
$$
\bs\lambda \mapsto D_{n}(1,{\bs \lambda})\qquad \textrm{and}\qquad \bs\lambda\mapsto D_{1,n}(1,
    \lambda_i),\quad 1\leq i\leq p
$$
are uniformly bounded on every compact subset of $\mC^p$.
\item Let $\bs \lambda\in \mC^p$ be fixed. For every $\delta>0$, there
  exists $r>0$ such that
\begin{eqnarray*}
&&\sup_n\sup_{z\in B(0,r)}|D_{n}(z,\bs\lambda)-1|\quad <\quad \delta\ ,\\
&& \sup_n  \sup_{z\in B(0,r)}|D_{i,n}(z,\lambda_i)-1|\quad <\quad \delta\ ,\quad 1\leq i\leq p\ . 
\end{eqnarray*} 

\end{enumerate}
\end{prop}
Proof of Proposition \ref{prop:DvoisZero} is provided in Appendix \ref{proof-prop:DvoisZero}.

We introduce the following functions: 
\begin{eqnarray}
d_n: (z,\bs\lambda)&\mapsto& \det\left(1- zS_n(\bs\lambda,\bs\Delta_n)\right) -
\prod_{i=1}^p \det\left(1- zS_n(\lambda_i,\Delta_{i,n})\right)\ ,\label{eq:dn}\\
f_n: (z,\bs\lambda)&\mapsto&   \frac{D_{n}'(z,\bs\lambda)}{D_{n}(z,\bs\lambda)} -   
\sum_{i=1}^p\frac{D_{i,n}'(z,\lambda_i)}{D_{i,n}(z,\lambda_i)}\ ,\label{eq:fn}
\end{eqnarray}
where $'$ denotes the derivative with respect to $z\in \mC$. We first
prove that $f_n$ goes to zero as $z\to 0$.

\subsubsection{Asymptotic study of  $f_n$ in a neighbourhood of $z=0$}\label{fn-to-zero}

In this section, we mainly consider the dependence of $f_n$ in $z\in
\mC$ while $\bs \lambda\in \mC^p$ is kept fixed. We therefore drop the
dependence in $\bs \lambda$ to lighten the notations. Equality
\eqref{eq:logdet} yields:
\begin{equation}\label{eq:Dn}
\frac{D_{n}'(z)}{D_{n}(z)} = - \sum_{k=0}^{\infty}T_{n}(k+1) z^k
\quad \textrm{and}\quad 
\frac{D_{i,n}'(z)}{D_{i,n}(z)} = - \sum_{k=0}^{\infty}T_{i,n}(k+1) z^k\quad (1\leq i\leq p)
\end{equation}
where $'$ denotes the derivative with respect to $z\in \mC$ and
$T_{n}(k)$ and $T_{i,n}(k)$ are as in \eqref{eq:trace}, respectively
defined by:
\begin{eqnarray}
  T_{n}(k) &:=& \int_{\mR^k}S_n(x_1,x_2)S_n(x_2,x_3)\cdots S_n(x_k,x_1)dx_1\cdots dx_k\:,
\label{eq:Tn}\\
  T_{i,n}(k) &:=& \int_{\mR^k}S_{i,n}(x_1,x_2)S_{i,n}(x_2,x_3)\cdots S_{i,n}(x_k,x_1)dx_1\cdots dx_k\:.
\label{eq:Tin}
\end{eqnarray}

Recall that $D_n$ and $D_{i,n}$ are entire functions (of $z\in \mC$).
The functions $\dsp \frac {D_n'}{D_n}$ and $\dsp \frac
{D_{i,n}'}{D_{i,n}}$ admit a power series expansion around zero given
by \eqref{eq:Dn}. Therefore, the same holds true for $f_n(z)$, moreover:

\begin{lemma}
\label{lem:fn}
  Define $R$ as in Proposition \ref{prop:ray}. For $n$ large enough, $f_n(z)$ defined by~(\ref{eq:fn})
is holomorphic on $B(0,R):=\{ z\in \mathbb{C},\ |z|<R\}$, and converges uniformly to zero as $n\to\infty$
on each compact subset of $B(0,R)$.
\end{lemma}

\begin{proof}
Denote by $\xi_{i}^{(n)}(x):=\lambda_i {\bs 1}_{\Delta_{i,n}}(x)$ and
recall that $T_n(k)$ is defined by \eqref{eq:Tn}. Using the identity
\begin{equation}\label{eq:identite}
\prod_{m=1}^k \sum_{i=1}^p a_{i m} = \sum_{\sigma\in \{1,\cdots\!,p\}^k} \prod_{m=1}^k a_{\sigma(m) m}, 
\end{equation}
where $a_{im}$ are complex numbers, $T_n(k)$ writes ($k\geq 2$):
\begin{eqnarray}
  T_{n}(k) &=& \int_{\mR^k}\left(\prod_{m=1}^{k} \sum_{i=1}^p  \xi_i^{(n)}(x_m)\right)K_n(x_1,x_2)\cdots K_n(x_k,x_1)dx_1\cdots dx_k\:,\nonumber\\
  &=&\sum_{\sigma\in\{1,\cdots, p\}^k} j_{n,k}(\sigma)\:,
\label{eq:Tndev}
\end{eqnarray}
where we defined
\begin{equation}
\label{eq:defjnk}
j_{n,k}(\sigma) := \int_{\mR^k}\left(\prod_{m=1}^{k}\xi_{\sigma(m)}^{(n)}(x_m)\right)K_n(x_1,x_2)\cdots K_n(x_k,x_1)dx_1\cdots dx_k\:.
\end{equation}
We split the sum in the right-hand side into two subsums. The first
is obtained by gathering the terms with $k$-uples $\sigma=(i,i,\cdots ,i)$ for $1\leq i\leq p$ and writes:
$$
\sum_{i=1}^p \int_{\mR^k}\left(\prod_{m=1}^{k}\lambda_i {\bs 1}_{\Delta_{i,n}}(x_m)\right)K_n(x_1,x_2)\cdots K_n(x_k,x_1)dx_1\cdots dx_k = \sum_{i=1}^p T_{i,n}(k)\:,
$$
where $T_{i,n}(k)$ is defined by \eqref{eq:Tin}. The remaining sum
consists of those terms for which there exists at least one
couple $(m,\ell)\in\{1,\cdots,k\}^2$ such that
$\sigma(m)\neq\sigma(\ell)$. Let
$$
{\cal S} = \left\{ \sigma \in \{1,\cdots, p\}^k \: :\: \exists
  (m,\ell)\in\{1,\cdots,k\}^2, \sigma(m)\neq \sigma(\ell) \right\}\:,
$$
we obtain $T_n(k) = \sum_{i=1}^p T_{i,n}(k) + s_n(k)$ where
$$
s_n(k) := \sum_{\sigma\in{\cal S}} j_{n,k}(\sigma)\:,
$$
for each $k\geq 2$.
For each $q\in\{1,\dots,k-1\}$, denote by $\pi_q$ the following permutation for any $k$-uplet $(a_1,\dots, a_k)$:
$$
\pi_q(a_1,\dots, a_k) = (a_q,a_{q+1},\dots,a_k,a_1,\dots,a_{q-1})\:.
$$
In other words, $\pi_q$ operates a circular shift of $q-1$ elements to the left.
Clearly, any $k$-uple $\sigma\in{\cal S}$ can be written as $\sigma= \pi_q(m,\ell,\tilde \sigma)$
for some $q\in \{1,\dots, k-1\}$, $(m,\ell)\in\{1,\dots,p\}$ such that $m\neq \ell$, and 
$\tilde\sigma \in\{1,\dots, p\}^{k-2}$. Thus, 
\begin{eqnarray*}
  |s_{n}(k)| &\leq& \sum_{q=1}^{k-1} \sum_{\substack{(m,\ell)\in\{1\cdots p\}^2 \\  m \neq \ell}}\sum_{\tilde \sigma \in\{1\cdots p\}^{k-2}} |j_{n,k}(\pi_q(m,\ell,\tilde\sigma))| \:.
\end{eqnarray*}
From~(\ref{eq:defjnk}), function $j_{n,k}$ is invariant up to any
circular shift $\pi_q$, so that $j_{n,k}(\sigma)$ coincides with
$j_n(m,\ell,\tilde \sigma)$ for any $\sigma= \pi_q(m,\ell,\tilde
\sigma)$ as above. Therefore,


\begin{eqnarray*}
  |s_{n}(k)|& \leq & \sum_{q=1}^{k-1} \sum_{\substack{(m,\ell)\in\{1\cdots p\}^2 \\  m \neq \ell}}\sum_{\tilde \sigma \in\{1\cdots p\}^{k-2}} |j_{n,k}(m,\ell,\tilde\sigma)|\\
&\le & k \sum_{\substack{(m,\ell)\in\{1\cdots p\}^2 \\  m \neq \ell}}
  \sum_{\tilde \sigma \in\{1\cdots p\}^{k-2}} \int_{\mR^k}|\xi_{m}^{(n)}(x_{1})\xi_{\ell}^{(n)}(x_{2})\xi_{\tilde \sigma(1)}^{(n)}(x_{3})\cdots \xi_{\tilde \sigma(k-2)}^{(n)}(x_{k})| \\
&& \quad ×   |K_n(x_1,x_2)\cdots K_n(x_k,x_1)|dx_1\cdots dx_k
\end{eqnarray*}
The latter writes
\begin{eqnarray*}
|s_{n}(k)| &=& k \sum_{\substack{1\leq m,\ell\leq p\\ m\neq \ell}}
 \int_{\Delta_{m,n}× \Delta_{\ell,n}} \left| K_n(x_1,x_2) \xi_{m}^{(n)}(x_{1})\xi_{\ell}^{(n)}(x_{2})\right| \\
 \lefteqn{ × \left(\int_{\mR^{k-2}}{\sum_{\tilde \sigma\in\{1\cdots p\}^{k-2}}\left|\xi_{\tilde \sigma(1)}^{(n)}(x_{3})\cdots \xi_{\tilde \sigma(k-2)}^{(n)}(x_{k})\right|\left|
 K_n(x_2,x_3)\cdots K_n(x_k,x_1)\right|dx_3\cdots dx_k }\right) dx_1dx_2}\\
&=& k \sum_{\substack{1\leq m,\ell\leq p \\ m\neq \ell}}
 \int_{\Delta_{m,n}× \Delta_{\ell,n}} \left|K_n(x_1,x_2)× \sum_{i=1}^p\xi_{i}^{(n)}(x_{1})\right|× \sum_{i=1}^p \left|\xi_{i}^{(n)}(x_{2})\right| \\
 \lefteqn{ ×  \left(\int_{\mR^{k-2}}{\sum_{\tilde \sigma\in\{1\cdots p\}^{k-2}}\left|\xi_{\tilde \sigma(1)}^{(n)}(x_{3})\cdots \xi_{\tilde \sigma(k-2)}^{(n)}(x_{k})\right|
 \left|K_n(x_2,x_3)\cdots K_n(x_k,x_1)\right|dx_3\cdots dx_k}\right) dx_1dx_2.}
 \end{eqnarray*}



It remains to notice that 
\begin{eqnarray*}
  \lefteqn{\sum_{i=1}^p \left|\xi_{i}^{(n)}(x_{2})\right|
    \int_{\mR^{k-2}}{\sum_{\tilde \sigma\in\{1\cdots p\}^{k-2}}\prod_{m=3}^k\left|\xi_{\tilde \sigma(m-2)}^{(n)}(x_{m})\right|
      \left|K_n(x_2,x_3)\cdots K_n(x_k,x_1)\right|dx_3\cdots dx_k}}\\
  &\stackrel{(a)}{=}&\sum_{i=1}^p \left|\xi_{i}^{(n)}(x_{2})\right| \int_{\mR^{k-2}} \left(\prod_{m=3}^k \sum_{i=1}^{p} 
\left|\xi_{i}^{(n)}(x_{m})\right| \right)\left|K_n(x_2,x_3)\cdots K_n(x_k,x_1)\right|dx_3\cdots dx_k\\
&=& \int_{\mR^{k-2}} |S_n(x_2,x_3) S_n(x_3,x_4)\cdots S_n(x_k,x_1)| dx_3\cdots dx_k\\
&\stackrel{(b)}=& |S_n|^{(k-1)}(x_2,x_1)\ ,
\end{eqnarray*}
where $(a)$ follows from \eqref{eq:identite}, and $(b)$ from the mere
definition of the iterated kernel \eqref{eq:def-iterated}. Thus, for
$k\geq 2$, the following inequality holds true:
\begin{equation}
  |s_n(k)|\leq k
\sum_{\substack{1\leq m,\ell\leq p \\ 
m\neq \ell}} \int_{\Delta_{m,n}× \Delta_{\ell,n}} |S_n(x_1,x_2)|× |S_n|^{(k-1)}(x_2,x_1) dx_1dx_2 \ .
\label{eq:sn}
\end{equation}
For $k=1$, let $s_n(1)=0$ so that equation $T_{n}(k) = \sum_i
T_{i,n}(k)+ s_n(k)$ holds for every $k\geq1$.  

According to \eqref{eq:fn}, $f_n(z)$ writes:
$$
f_n(z)= - \sum_{k=1}^\infty s_n(k+1)z^k\ .
$$
Let us now prove that $f_n(z)$ is well-defined on the desired
neighbourhood of zero and converges uniformly to zero as $n\to\infty$. Let $\beta>R^{-1}$, then 
Propositions \ref{prop:ray} and \ref{prop:estimees-M}
yield:
\begin{eqnarray*}
  |s_n(k)|&\leq& k \sum_{\substack{1\leq m,\ell\leq p \\ m\neq \ell}} \int_{\Delta_{m,n}× \Delta_{\ell,n}} |S_n(x,y)| |S_n|^{(k-1)}(y,x) dxdy\ ,\\ 
  &\leq& k \sum_{\substack{1\leq m,\ell\leq p \\ m\neq \ell}} {\cal M}_{m× \ell,n} {\cal M}_{\ell× m,n}^{(k-1)} |\Delta_{m,n}| |\Delta_{\ell,n}|\ ,\\
  &\leq& k \: \beta^{k-2}\sum_{\substack{1\leq m,\ell\leq p \\ m\neq \ell}} C^2
  n^{\left(1-\frac{\kappa_m +\kappa_{\ell}}2\right)}n^{\left(1+\epsilon-\frac{\kappa_m +\kappa_{\ell}}2\right)} n^{-(\kappa_m+\kappa_{\ell})}× |\Delta_m \Delta_{\ell}|
  \ ,  \\
  &\leq& k \: \beta^{k-2}\sum_{\substack{1\leq m,\ell\leq p \\ m\neq \ell}} \frac{C^2\ |\Delta_{m}\Delta_{\ell}|}{n^{2(\kappa_m+\kappa_{\ell}-1)-\epsilon}}  \ ,\\
  &\stackrel{(a)}\leq& k \: \beta^{k-2} × \left( \max_{1\leq m\leq p} |\Delta_m|\right)^2 × \frac{p(p-1)C^2}{n^{\frac 23 -\epsilon}}\ ,
\end{eqnarray*}
where $(a)$ follows from the fact that
$\kappa_m+\kappa_{\ell}-1\geq \frac 13 $.  Clearly,
the power series $\sum_{k=1}^\infty (k+1)\beta^{k-1} z^k$ converges for
$|z|<\beta^{-1}$. As $\beta^{-1}$ is arbitrarily lower than $R$, this
implies that $f_n(z)$ is holomorphic in $B(0,R)$. Moreover, 
for each compact subset $K$ included in the open disk
$B(0,\beta^{-1})$ and for each $z\in K$,
$$
|f_n(z)| \leq \left( \sum_{k=1}^\infty  (k+1)\beta^{k-1} (\sup_{z\in K} |z| )^k\right)× 
\left( \max_{1\leq m\leq p} |\Delta_m|\right)^2 × \frac{p(p-1)C^2}{n^{\frac 23 -\epsilon}}\ .
$$
The right-hand side of the above inequality converges to zero as
$n\to\infty$. Thus, the uniform convergence of $f_n(z)$ to
zero on $K$ is proved; in particular, as $\beta^{-1}<R$, $f_n(z)$
converges uniformly to zero on $B(0,R)$. Lemma~\ref{lem:fn} is proved.

\end{proof}
\subsubsection{Convergence of $d_n$ to zero as $n\to \infty$}

In this section, $\bs \lambda\in \mC^p$ is fixed. We therefore drop
the dependence in $\bs\lambda$ in the notations. Consider function $F_n$
defined by:
\begin{equation}
  \label{eq:Fn}
  F_n(z) := \log\frac{D_{n}(z)}{\prod_{i=1}^pD_{i,n}(z)}\ ,
\end{equation}
where $\log$ corresponds to the principal branch of the logarithm and
$D_n$ and $D_{i,n}$ are defined in \eqref{eq:Dn}. As
$D_n(0)=D_{i,n}(0)=1$, there exists a neighbourhood of zero where
$F_n$ is holomorphic. Moreover, using Proposition
\ref{prop:DvoisZero}-3), one can prove that there exists a
neighbourhood of zero, say $B(0,\rho)$, where $F_n(z)$ is a normal
family. Assume that this neighbourhood is included in $B(0,R)$, where
$R$ is defined in Proposition \ref{prop:ray} and notice that in this
neighbourhood, $F'_n(z)=f_n(z)$ as defined in \eqref{eq:fn}. Consider
a compactly converging subsequence $F_{\phi(n)} \to F_{\phi}$ in
$B(0,\rho)$ (by compactly, we
mean that the convergence is uniform over any compact set ${\bs
  \Lambda} \subset B(0,\rho)$), then one has in particular $F'_{\phi(n)}(z) \to
F'_{\phi}$ but $F'_{\phi(n)}(z)=f_{\phi(n)}(z) \to 0$.
Therefore, $F_{\phi}$ is a constant over $B(0,\rho)$, in particular,
$F_{\phi}(z)=F_{\phi}(0)=0$. We have proved that every converging
subsequence of $F_n$ converges to zero in $B(0,\rho)$. This yields the
convergence (uniform on every compact of $B(0,\rho)$) of $F_n$ to zero
in $B(0,\rho)$. This yields the existence of a neighbourhood of zero,
say $B(0,\rho')$ where:
\begin{equation}\label{conv-ratio}
\frac{D_{n}(z)}{\prod_{i=1}^p D_{i,n}(z)} \xrightarrow[n\to\infty]{} 1
\end{equation}
uniformly on every compact of $B(0,\rho')$. Recall that $d_n(z)=D_{n}(z) - \prod_{i=1}^pD_{i,n}(z)$.

Combining \eqref{conv-ratio} with Proposition \ref{prop:DvoisZero}-3) yields
the convergence of $d_n(z)$ to zero in a small neighbourhood of zero.
Now, Proposition \ref{prop:DvoisZero}-1) implies that $d_n(z)$ is a normal
family in $\mathbb{C}$. In particular, every subsequence $d_{\phi(n)}$
compactly converges to a holomorphic function which coincides with 0 in a small
neighbourhood of the origin, and thus is equal to 0 over $\mathbb{C}$.
We have proved that 
$$
d_n(z) \xrightarrow[n\to\infty]{} 0,\qquad \forall z\in \mathbb{C}\ ,
$$
with $\bs\lambda\in \mC^p$ fixed.

\subsubsection{Convergence of the partial derivatives  
  of $\bs\lambda \mapsto d_n(1,\bs\lambda)$ to zero}

In order to establish Theorem \ref{theo:main}, we shall rely on Proposition
\ref{prop:representation} where the probabilities of interest are
expressed in terms of partial derivatives of Fredholm determinants.
We therefore need to establish that the partial derivatives of
$d_n(1,{\bs \lambda})$ with respect to $\bs\lambda$ converge to zero as well.  This is the aim of this
section.

In the previous section, we have proved that $\forall
(z,\bs\lambda)\in \mC^{p+1}$, $d_n(z,{\bs \lambda})\to 0$ as $n\to\infty$. In particular,
$$
d_n(1,{\bs \lambda})\to 0,\quad \forall {\bs \lambda}\in \mathbb{C}^p\ .
$$ 
We now prove the following facts (with a slight abuse of notation, denote $d_n({\bs \lambda})$ instead of 
$d_n(1,{\bs \lambda})$):
\begin{enumerate}
\item As a function of ${\bs \lambda}\in \mathbb{C}^p$, $d_n({\bs \lambda})$ is holomorphic.
\item The sequence $\left(\bs\lambda\mapsto  d_n(\bs\lambda) \right)_{n\geq 1}$ is a normal family on $\mathbb{C}^p$.  
\item The convergence $d_n({\bs \lambda})\to 0$ is uniform
  over every compact set ${\bs \Lambda} \subset \mathbb{C}^p$.
\end{enumerate}

Proof of Fact 1) is straightforward and is thus omitted. Proof of Fact
2) follows from Proposition \ref{prop:DvoisZero}-2).  Let us now turn
to the proof of Fact 3). As $(d_n)$ is a normal family, one can
extract from every subsequence a compactly converging one in
$\mathbb{C}^p$ (see for instance \cite[Theorem 1.13]{Ran86}). But for every ${\bs \lambda}\in \mathbb{C}^p$,
$d_n({\bs \lambda})\to 0$, therefore every compactly
converging subsequence converges toward 0. In particular, $d_n$ itself
compactly converges toward zero, which proves Fact 3).

In order to conclude the proof, it remains to apply standard results
related to the convergence of partial derivatives of compactly
converging holomorphic functions of several complex variables, as for
instance \cite[Theorem 1.9]{Ran86}. As $d_n({\bs \lambda})$
compactly converges to zero, the following convergence holds true: Let
$(\ell_1,\cdots, \ell_p)\in \mathbb{N}^p$, then
$$
\forall {\bs \lambda}\in \mathbb{C}^p,\qquad 
          \left(\frac{\partial}
 {\partial\lambda_1}\right)^{\ell_1}\cdots\left(\frac{\partial}{\partial\lambda_p}\right)^{\ell_p} d_n({\bs \lambda}) \xrightarrow[n\to\infty]{} 0\ .
$$
This, together with Proposition \ref{prop:representation}, completes the proof of Theorem \ref{theo:main}.

\appendix
\subsection{Proof of Proposition \ref{prop:representation}}\label{proof:representation}
Denote
by $E_n({\bs \ell},\bs \Delta)$ the probability that for every  $i\in \{1,\cdots, p\}$, the set
$\Delta_i$ contains exactly $\ell_i$ eigenvalues:
\begin{equation}\label{def-En}
E_n({\bs \ell},\bs \Delta)=\PP{{\cal N}(\Delta_1)=\ell_1,\cdots,{\cal
    N}(\Delta_p)=\ell_p}\ .
\end{equation} 
Let ${\cal P}_n(m)$ be the set of subsets of $\{1,\cdots,n\}$ with exactly $m$ elements.
If $A\in {\cal P}_n(m)$, then ${A}^c$ is its complementary
subset in $\{1,\cdots,n\}$. The mere definition of $E_n(\bs \ell,\bs
\Delta)$ yields:
$$
E_n(\bs \ell,\bs \Delta) = \int_{\mR^n} \sum_{\substack{(A_1,\cdots,
    A_p) \in\\ {\cal P}_n(\ell_1)× \cdots × {\cal P}_n(\ell_p)}}
\prod_{k=1}^p \left\{ \prod_{i\in
    A_k}\mathbf{1}_{\Delta_k}(x_i)\prod_{j\in
    A_k^c}(1-\mathbf{1}_{\Delta_k}(x_j))\right\} p_n(x_1\cdots
x_n)dx_1\cdots dx_n
$$
Using the following formula:
$$
\frac{1}{\ell!}\left(-\frac{d}{d\lambda}\right)^{\ell}
\prod_{i=1}^n(1-\lambda \alpha_i) = \sum_{A\in {\cal
    P}_n(\ell)}\prod_{i\in A}\alpha_i\prod_{j\in A^c}(1-\lambda
\alpha_j)\ ,
$$
we obtain:
\begin{eqnarray*}
  E_n(\bs \ell,\bs \Delta) &=& 
\frac{1}{\ell_1!\cdots \ell_p!}\left(-\frac{\partial}{\partial\lambda_1}\right)^{\ell_1}
\cdots\left(-\frac{\partial}{\partial\lambda_p}\right)^{\ell_p}
  \Gamma(\bs\lambda,\bs\Delta)\bigg|_{\lambda_1=\cdots=\lambda_p=1}
\end{eqnarray*}
where
$$
\Gamma(\bs\lambda,\bs\Delta)=\int_{\mR^n} \prod_{i=1}^n(1-\lambda_1\mathbf{1}_{\Delta_1}(x_i))
\cdots (1-\lambda_p\mathbf{1}_{\Delta_p}(x_i))\, p_n(x_1\cdots x_n)\,dx_1\cdots dx_n\ .
$$
Expanding the inner product and using the fact that the $\Delta_k$'s are pairwise disjoint yields:
$$
(1-\lambda_1\mathbf{1}_{\Delta_1}(x))
\cdots (1-\lambda_p\mathbf{1}_{\Delta_p}(x))
= \left(1-\sum_{k=1}^p \lambda_k\mathbf{1}_{\Delta_k}(x)\right)\ .
$$
Thus
\begin{eqnarray*}
  \Gamma(\bs\lambda,\bs\Delta)&=&  \int_{\mR^n} \prod_{i=1}^n
  \left(1-\sum_{k=1}^p \lambda_k\mathbf{1}_{\Delta_k}(x_i)\right)\, p_n(x_1\cdots x_n)\, dx_1\cdots dx_n\ ,\\
  &\stackrel{(a)}=& 1+ \int_{\mR^n} \sum_{m=1}^n (-1)^m  \sum_{A\in {\cal P}_n(m)}\, \prod_{i\in A}\left(\sum_{k=1}^p 
    \lambda_k\mathbf{1}_{\Delta_k}(x_i)\right)\  p_n(x_1\cdots x_n)\, dx_1\cdots dx_n\ ,\\
  &=& 1+ \sum_{m=1}^n (-1)^m  \sum_{A\in {\cal P}_n(m)}\int_{\mR^n}
  \prod_{i\in A}\left(\sum_{k=1}^p \lambda_k\mathbf{1}_{\Delta_k}(x_i)\right)\, p_n(x_1\cdots x_n)
\,dx_1\cdots dx_n\ ,\\
  &\stackrel{(b)} =& 1+ \sum_{m=1}^n (-1)^m  \binom{n}{m}\int_{\mR^n}
  \prod_{i=1}^m\left(\sum_{k=1}^p \lambda_k\mathbf{1}_{\Delta_k}(x_i)\right)\, p_n(x_1\cdots x_n)\,dx_1\cdots dx_n\ ,\\
  &\stackrel{(c)} =& 1+ \sum_{m=1}^n  \frac{(-1)^m}{m!}\int_{\mR^m}
\prod_{i=1}^m\left(\sum_{k=1}^p \lambda_k\mathbf{1}_{\Delta_k}(x_i)\right)  \det\left\{  K_n(x_i,x_j) \right\}_{1\leq i,j\leq  m}
dx_1\cdots dx_m\ ,
\end{eqnarray*}
where $(a)$ follows from the expansion of $\prod_i \left(1-\sum_k
  \lambda_k\mathbf{1}_{\Delta_k}(x_i)\right)$, $(b)$ from the fact
that the inner integral in the third line of the previous equation
does not depend upon $E$ due to the invariance of $p_n$
with respect to any permutation of the $x_i$'s, and $(c)$ follows from
the determinantal representation \eqref{eq:marginals}.

Therefore, $\Gamma(\bs\lambda,\bs\Delta)$ writes:
\begin{equation}
  \Gamma(\bs\lambda,\bs\Delta)
= 1+ \sum_{m=1}^n   \frac{(-1)^m}{m!}\int_{\mR^m}
 \det\left\{  S_n(x_i,x_j;\bs\lambda,\bs\Delta) \right\}_{1\leq i,j\leq m}dx_1\cdots dx_m \label{eq:Gld}
\end{equation}
where $S_n(x,y;\bs\lambda,\bs\Delta)$ is the kernel defined in
(\ref{eq:Sn}).  As the operator $S_n(\bs \lambda,\bs\Delta)$ has
finite rank $n$, (\ref{eq:Gld}) coincides with the Fredholm determinant
$\det(1-S_n(\bs \lambda,\bs\Delta))$ (see \cite{Tri57} for details). Proof of Proposition
\ref{prop:representation} is completed.

\subsection{Proof of Proposition~\ref{prop:ray}}\label{proof-prop:ray}

In the sequel, $C>0$ will be a constant independent from $n$, but
whose value may change from line to line.

First consider the case $i=j$.
Denote by $S_{\mu_i}(x,y)$ the following limiting kernel:
$$
S_{\mu_i}(x,y) :=\left\{
  \begin{array}{ll}
   {\dsp  \frac{\sin \pi\rho(\mu_i)(x-y)}{\pi(x-y)}}& \textrm{ if } -2 < \mu_i < 2 \\
    {\dsp \frac{Ai(x)Ai'(y)-Ai(y)Ai'(x)}{x-y}}&\textrm{ if } \mu_i=2, \\
    {\dsp \frac{Ai(-x)Ai'(-y)-Ai(-y)Ai'(-x)}{-x+y}}&\textrm{ if } \mu_i=-2, \\
  \end{array}\right.
$$
Proposition \ref{prop:convergence-kernel} implies that
$n^{-\kappa_i}K_n(\mu_i+x/n^{\kappa_i},\mu_i+y/n^{\kappa_i})$
converges uniformly to $S_{\mu_i}(x,y)$ on every compact subset of
$\mR^2$, where $\kappa_i$ is defined by~(\ref{eq:deltai}). Moreover, 
$S_{\mu_i}(x,y)$ being bounded on every compact subset of $\mR^2$, there
exists a constant $C_i$ such that:
\begin{eqnarray}
  {\cal M}_{i× i,n}(\bs \Lambda)&=& \left( 
\sup_{\bs \lambda \in \bs \Lambda}|\lambda_i| \right) 
\sup_{(x,y)\in \Delta_{i,n}^2} \left| K_n\left(x,y\right)\right|
\quad =\quad 
\left( \sup_{\bs \lambda \in \bs \Lambda}|\lambda_i| \right) \sup_{(x,y)\in \Delta_{i}^2} \left| K_n\left(\mu_i+\frac{x}{n^{\kappa_i}},\mu_i+\frac{y}{n^{\kappa_i}}\right)\right| \nonumber\\
  &\leq& \left( \sup_{\bs \lambda \in \bs \Lambda} |\lambda_i|\right) n^{\kappa_i} \left(\sup_{(x,y)\in \Delta_{i}^2} \left| \frac{1}{n^{\kappa_i}} 
      K_n\left(\mu_i+\frac{x}{n^{\kappa_i}},\mu_i+\frac{y}{n^{\kappa_i}}\right) -S_{\mu_i}(x,y)\right| +
    \sup_{(x,y)\in \Delta_{i}^2}\left| S_{\mu_i}(x,y)\right|\right) \nonumber\\
  &\leq& n^{\kappa_i} C_i\:,
\label{eq:boundMin}
\end{eqnarray}
It remains to take $R$ as $R^{-1} = \max (C_1,\cdots, C_p)$ to get the pointwise or uniform estimate.\\

Consider now the case where $i\neq j$. Using notation $\kappa_i$,
inequalities~(\ref{eq:psibulk}) and~(\ref{eq:psiedge}) can be conveniently merged as follows:
There exists a constant $C$ such that for $1\leq i\leq p$, 
\begin{equation}\label{eq:ineg-generale}
\sup_{x\in\Delta_{i,n}} \left|\psi_{n-k}^{(n)}(x)\right| \leq n^{\frac{1-\kappa_i}{2}}C\:.
\end{equation}
For $n$ large enough, we obtain, using~\eqref{eq:kernel}:
\begin{eqnarray*}
  {\cal M}_{i× j,n} (\bs \Lambda) &\stackrel{(a)}{\leq}&  \left( 
\sup_{\bs \lambda \in \bs \Lambda}|\lambda_i| \right) \sup_{(x,y)\in \Delta_{i,n}×\Delta_{j,n}} \frac{|\psi_n^{(n)}(x)| |\psi_{n-1}^{(n)}(y)|+ |\psi_n^{(n)}(y)| |\psi_{n-1}^{(n)}(x)|}{|x-y|}\ ,\\
  &\stackrel{(b)}{\leq}&   \left( \sup_{\bs \lambda \in \bs \Lambda} |\lambda_i|\right) n^{\frac{1-\kappa_i}{2}+\frac{1-\kappa_j}{2}}\frac{2C^2}{\inf_{(x,y)\in \Delta_{i,n}×\Delta_{j,n}}|x-y|}\ ,\\
  &\stackrel{(c)}{\leq}&  C \: n^{1-\frac{\kappa_i+\kappa_j}{2}}\ ,
\end{eqnarray*}
where $(a)$ follows from \eqref{eq:kernel}, $(b)$ from \eqref{eq:ineg-generale} and $(c)$ from the fact that 
$$
\liminf_{n\to\infty} \inf_{(x,y)\in \Delta_{i,n}×\Delta_{j,n}}|x-y| = |\mu_i -\mu_j|>0.
$$
Proposition \ref{prop:ray} is proved.

\subsection{Proof of Proposition \ref{prop:estimees-M}}\label{proof-prop:estimees-M}

Let ${\bs \Lambda}= \{ \bs \lambda\}$ be fixed. We drop, in the rest
of the proof, the dependence in $\bs \lambda$ in the notations. The
mere definition of $|S_n|^{(k)}$ yields:
\begin{eqnarray*}
  0\leq |S_n|^{(k)}(x,y) &\leq& \int_{\mR}|S_n(x,u)|× |S_n|^{(k-1)}(u,y)du\\
  &=& \sum_{i=1}^p \int_{\Delta_{i,n}}|S_n(x,u)|× |S_n|^{(k-1)}(u,y)du 
\end{eqnarray*}
From the above inequality, the following is straightforward:
\begin{eqnarray*}
\forall (x,y)\in\Delta_{m,n}×\Delta_{\ell,n},\;\;
|S_n|^{(k)}(x,y) &\leq& \sum_{i=1}^p |\Delta_{i,n}| {\cal M}_{m× i,n} {\cal M}_{i× \ell,n}^{(k-1)}\ .
\end{eqnarray*}
Using Proposition~\ref{prop:ray}, we obtain:
\begin{eqnarray}\label{eq:Munk} {\cal M}_{m× \ell,n}^{(k)}&\leq& R^{-1} {\cal M}_{m×
    \ell,n}^{(k-1)} + \alpha  \sum_{i\neq m}
  n^{(1- \frac{\kappa_m+ 3\kappa_i}2)} {\cal M}_{i× \ell,n}^{(k-1)}\ ,
\end{eqnarray}
where $\alpha:=\max(C|\Delta_1|,\cdots, C|\Delta_p|)$.
Now take $\beta> R^{-1}$ and $\epsilon\in (0,\frac 13)$.
Property~\eqref{eq:Knbd} holds for $k=1$ since 
$$
{\cal M}_{m× m,n}\leq R^{-1} n^{\kappa_m} \leq \beta n^{\kappa_m}\quad \textrm{and}\quad 
{\cal M}_{m× \ell,n} \leq C n^{\left(1-\frac{\kappa_m +\kappa_{\ell}}2\right)}
\leq C n^{\left(1 + \epsilon -\frac{\kappa_m +\kappa_{\ell}}2\right)}
$$
for every $m\neq \ell$ by Proposition ~\ref{prop:ray}. Assume that the same holds 
at step $k-1$. 

Consider first the case where $m=\ell$. Eq.~(\ref{eq:Munk}) becomes
\begin{eqnarray*}
{\cal M}_{m× m,n}^{(k)} &\leq&  R^{-1}  \beta^{k-1} n^{\kappa_m} +\alpha C \beta^{k-2} 
\sum_{i\neq m} n^{(1- \frac{\kappa_m}2 -\frac{3\kappa_i}2)} n^{(1+\epsilon - \frac{\kappa_i}2 -\frac{\kappa_{m}}2)} \\
 &\leq& \beta^{k} n^{\kappa_m} \left(\frac{R^{-1}}{\beta} +\sum_{i\neq m} \frac{\alpha C}{\beta^2} 
n^{\left( 2+\epsilon -2\kappa_m - 2\kappa_i\right)}\right)  \\
 &\leq& \beta^{k} n^{\kappa_m}\qquad \textrm{for $n$ large enough}\ ,
\end{eqnarray*}
where the last inequality follows from the fact that
$2+ \epsilon -2\kappa_m -2 \kappa_i<0$, which implies that
$n^{2+ \epsilon -2\kappa_m -2 \kappa_i}\to 0$, which
in turn implies that the term inside the parentheses is lower than one
for $n$ large enough.  

Now if $m\neq \ell$, Eq.~(\ref{eq:Munk}) becomes:
\begin{eqnarray*} {\cal M}_{m× \ell,n}^{(k)} &\leq& R^{-1} C
    \beta^{k-2} n^{\left(1+\epsilon -\frac{\kappa_{\ell}+\kappa_m}2\right)}
    +\alpha\beta^{k-1} n^{\left(1 -\frac{\kappa_{\ell}+\kappa_m}2\right)}       
  +\sum_{i\neq m,\ell} C\alpha\beta^{k-2} n^{\left( 1-\frac{\kappa_m+ 3\kappa_i}2\right)} 
n^{\left(1+\epsilon -\frac{\kappa_i+\kappa_{\ell}}2\right)} \\
  &=& C \beta^{k-1} n^{\left(1+\epsilon -\frac{\kappa_{\ell}+\kappa_m}2\right)}
\left(
    \frac{R^{-1}}{\beta}+\frac{\alpha}{C n^{\epsilon}}
    +\frac{\alpha }{\beta}\sum_{i\neq m,\ell} n^{1-2\kappa_i} \right)\\
  &\leq& C \beta^{k-1} n^{\left(1+\epsilon -\frac{\kappa_{\ell}+\kappa_m}2\right)}
\left(
    \frac{R^{-1}}{\beta}+\frac{\alpha}{C n^{\epsilon}}
    +\frac{\alpha p^2}{\beta n^{\frac 13}} \right)
\\
  &\leq& C
    \beta^{k-1} n^{\left(1+\epsilon -\frac{\kappa_{\ell}+\kappa_m}2\right)} \ ,
\end{eqnarray*}
where the last inequality follows from the fact that the term inside
the parentheses is lower than one for $n$ large enough.
Therefore,~(\ref{eq:Knbd}) holds for each $k\geq 1$ and for $n$ large enough.

\subsection{Proof of Proposition \ref{prop:DvoisZero}}\label{proof-prop:DvoisZero}

Define $U_n(k,\bs\lambda):=\int_{\mR^k} \left|\det \left\{
    S_n(x_i,x_j;\bs\lambda)\right\}_{i,j=1\cdots k}\right|dx_1\cdots
dx_k$.  Using Hadamard's inequality,
\begin{eqnarray}
 U_n(k,\bs\lambda) &\leq& \int_{\mR^k} \prod_{i=1}^k \sqrt{\sum_{j=1}^k|S_n(x_i,x_j;\bs\lambda)|^2 }dx_1\cdots dx_k \nonumber\\
&\leq& \int_{\mR^k} \prod_{i=1}^k \sqrt{\sum_{j=1}^k\left|\sum_{m=1}^p\lambda_m{\bs 1}_{\Delta_{m,n}}(x_i)\right|^2|K_n(x_i,x_j)|^2 }dx_1\cdots dx_k \nonumber
\end{eqnarray}
Therefore, 
\begin{eqnarray}
  U_n(k,\bs\lambda) &\leq& \int_{\mR^k} \prod_{i=1}^k \left(\sum_{m=1}^p \lambda_m {\bs 1}_{\Delta_{m,n}}(x_i)\right)\sqrt{\sum_{j=1}^k|K_n(x_i,x_j)|^2} dx_1\cdots dx_k \nonumber\\
  &=& \int_{\mR^k} \sum_{\sigma\in\{1\cdots p\}^k} \prod_{i=1}^k \lambda_{\sigma(i)} {\bs 1}_{\Delta_{\sigma(i),n}}(x_i)\sqrt{\sum_{j=1}^k|K_n(x_i,x_j)|^2} dx_1\cdots dx_k \nonumber\\
  &=& \sum_{\sigma\in\{1\cdots p\}^k}  \int_{\mR^k} \prod_{i=1}^k\sqrt{\sum_{j=1}^k| \lambda_{\sigma(i)} {\bs 1}_{\Delta_{\sigma(i),n}}(x_i)K_n(x_i,x_j)|^2} dx_1\cdots dx_k \:.\nonumber
\end{eqnarray}
In the above equation, integral $\int_{\mR^k}$ clearly reduces to an integral on the
set $\Delta_{\sigma(1),n}× \cdots × \Delta_{\sigma(p),n}$. Thus,
\begin{eqnarray}
  \sup_{\bs \lambda \in \Lambda} U_n(k,\bs \lambda) &\leq& \sum_{\sigma\in\{1\cdots p\}^k}  \int_{\Delta_{\sigma(1),n}× \cdots × \Delta_{\sigma(p),n}} \prod_{i=1}^k\sqrt{\sum_{j=1}^k {\cal M}_{\sigma(i)×\sigma(j)}^2}(\bs\Lambda) dx_1\cdots dx_k \nonumber\\
  &=& \sum_{\sigma\in\{1\cdots p\}^k} \prod_{i=1}^k\sqrt{\sum_{j=1}^k\left(|\Delta_{\sigma(i),n}| 
      {\cal M}_{\sigma(i)×\sigma(j)}(\bs\Lambda)\right)^2} \label{eq:unkbd}
\end{eqnarray} 
We now use Proposition~\ref{prop:ray} to bound the right-hand side.
Clearly, when $\sigma(i)=\sigma(j)$, Proposition~\ref{prop:ray}
implies that $|\Delta_{\sigma(i),n}|{\cal M}_{\sigma(i)×
  \sigma(i),n}(\bs\Lambda)\leq R_{\bs\Lambda}^{-1} \Delta_{\max}$, where
$\Delta_{\max}=\max_{1\leq i\leq p} |\Delta_i|$. This inequality still
holds when $\sigma(i)\neq \sigma(j)$ as a simple application of
Proposition~\ref{prop:ray}.  Therefore,
$$
\sup_{\bs\lambda\in \bs \Lambda} U_n(k,\bs\lambda) \leq \sum_{\sigma\in\{1,\cdots, p\}^k}
k^{\frac{k}{2}}\Delta_{\max}^k R_{\bs\Lambda}^{-k} =
\left(\frac{p\,\Delta_{\max}\sqrt{k}}{R_{\bs\Lambda}}\right)^k\:.
$$
Using this inequality, it is straightforward to show that the serie
$\sum_k \frac{\sup_{\bs\lambda\in \bs\Lambda}U_n(k,\bs\lambda)}{k!}z^k$ converges for 
every $z\in\mC$ and every compact set $\bs\Lambda$. Parts 1) and 2) of the proposition are proved. 
Based on the definition of $D_{n}(z,\bs\lambda)$ and $D_{i,n}(z,\lambda_i)$, we
obtain:
$$
\max\left(|D_{n}(z,\bs\lambda)-1|,|D_{i,n}(z,\lambda_i)-1|,\ 1\leq i\leq p \right) \leq 
|z|\sum_{k=1}^\infty \frac{|z|^{k-1}}{k!} U_n(k,\bs\lambda)\:,
$$
which completes the proof of Proposition~\ref{prop:DvoisZero}.

\bibliography{math}
\end{document}